\documentclass[12pt]{amsart}
\usepackage{latexsym, amsmath, amscd, amssymb, amsthm} 
\usepackage[english]{babel}
  \newtheorem{lm}{Lemma}
\newtheorem{te}{Theorem}
\begin{document}

\noindent

 \title{ A new formula for the generating function of  the numbers of  simple graphs } 

\author{Leonid Bedratyuk  and Anna Bedratyuk}
\address{Khmelnitskiy national university, Instytuts'ka, 11,  Khmelnits'ky, 29016, Ukraine}
\email{leonid.uk@gmal.com}

\begin{abstract}  By using an  approach of the invariant theory we obtain a new formula for the ordinary generating function of the numbers of the  simple graphs with $n$ nodes.

\end{abstract}
\maketitle


{\bf 1. Introduction.}
Let  $a_{n,i} $ be the number of simple graphs  with  $n$ vertices and  $k$ edges. Let  $$g_{n}(z)=\sum_{i=0}^{m} a_{n\!,\,i} z^i, m={n \choose 2},$$ be the ordinary generating function for the sequence  $\{a_{n,i} \}$, the OIES sequence $A008406$.  For the small  $n$ we have

\begin{align*}
&g_1(z)=1,\\
&g_2(z)=1+z, \\
&g_3(z)=1+z+z^2+z^3,\\
&g_4(z)= 1+z+2z^2+3z^3+2z^4+z^5+z^6.\\
\end{align*}
An expression for  $g_n(z)$ in terms of  group cycle index was found by  Harary in \cite{H4}. The result is based on Polya's efficient method for counting graphs, see  \cite{PP}  and \cite{H-P}.   Let $G$ be a permutation group acting on the set $[n]:=\{1,2, \ldots, n \}.$ It is 
well known that each permutation $\alpha$ in $G$ can be written uniquely as a product 
of disjoint cycles.    Let $j_i(\alpha)$ be the number 
of cycles of length $i, 1 \leq i \leq n$ in the disjoint cycle decomposition of $\alpha$. Then the 
cycle index of $G$  denoted $Z(G,s_1,s_2, \ldots,s_n),$ is the polynomial in the variables $s_1,s_2,\ldots,s_n$ defined by
$$
Z(G,s_1,s_2, \ldots,s_n)=\frac{1}{|G|}\sum_{\alpha \in G} \prod_{i=0}^n s_i^{j_i(\alpha)}.
$$

Denote by $[n]^{(2)}$  the set of  $2$-subsets of $[n].$ Let $S_n$ be a permutation group on the  set $[n].$ The 
pair group of $S_n$, denoted $S^{(2)}_n$ is the permutation group induced by $S_n$ which 
acts on $[n]^{(2)}$. Specifically, each permutation $\sigma \in S_n$  induces a permutation 
$\sigma' \in S^{(2)}_n$ such that for every element $ \{i,j\} \in [n]^{(2)}$ we  have 
$\sigma' \{i,j \} = \{\sigma i, \sigma j\}.$  

 In \cite{H4} F. Harary proved that  the generating function  $g_n(z)$
 is determined by substituting $1+z^k$ for each variable $s_k$ in the cycle index $Z(S^{(2)}_n,s_1,s_2, \ldots,s_n).$
 Symbolically 
 
$$
g_{n}(z)=Z(S_n^{(2)},1+z),
$$
where 
$$
Z(S_n^{(2)})=\frac{1}{n!} \sum_{j_1+2j_2+\cdots+n j_n=n}\frac{n!}{\prod\limits_{k=1}^n k^{j_k} j_k!} \prod_k s_{2k+1}^{k j_{2k+1}} \prod_k(s_k s_{2k}^{k-1})^{j_{2k}} s_k^{k \binom{j_k}{2}} \prod_{r<t}s^{(r,t)j_r j_r}_{[r,t]}. 
$$

In the paper by an  approach of the invariant theory we derive another formula for the generating function  $g_n(z)$. 

Let $\mathcal{V}_n$ be a vector space of weighted graphs on $n$ vertices over the field   $\mathbb{K}$, $\dim \mathcal{V}_n=m$. The group    $S^{(2)}_n$  acts naturally on  $\mathcal{V}_n$  by permutations of the basic vectors. Consider the corresponding action of the group $S^{(2)}_n$ on the algebra of polynomial functions   $\mathbb{K}[\mathcal{V}_n]$  and let  $\mathbb{K}[\mathcal{V}_n]^{S^{(2)}_n}$  be the corresponding algebra of invariants. Let  $\mathcal{V}^0_n$ be the set of simple graphs. The corresponding algebra of invariants  $k[\mathcal{V}^0_n]^{S^{(2)}_n}$  is a finite-dimensional vector space and can be expanded into the direct sum of its subspaces: 
$$
\mathbb{K}[\mathcal{V}_n^0]=(\mathbb{K}[\mathcal{V}_n^0])_0+(\mathbb{K}[\mathcal{V}_n^0])_1+\cdots+(\mathbb{K}[\mathcal{V}_n^0])_m.
$$

In the paper we have proved that    $\dim (\mathbb{K}[\mathcal{V}_n^0])_i=a_{n,i}$. Thus  the generating function  $g_n(z)$  coincides with the Poincar\'e series  $\mathcal{P}(\mathbb{K}[\mathcal{V}_n^0]^{S^{(2)}_n},z)$  of the algebra invariants  $\mathbb{K}[\mathcal{V}_n^0]^{S^{(2)}_n}.$

Let us identify the elements of the group   $S_n^{(2)}$ with    permutation $m \times m $ matrices and denote  $\text{{\rm \textbf{1}$_m$}}$ the identity  $m \times m $ matrix. In the paper we offer the following formula for  the generating function  $g_n(z)$:  

$$
g_{n}(z)=\frac{1}{n!} \sum_{\alpha  \in S_n^{(2)}} \frac{\det(\text{{\rm \textbf{1}$_m$}}-\alpha \cdot z^2)}{\det(\text{{\rm \textbf{1}$_m$}}-\alpha \cdot z)}.
$$

Also for the generating function  $m_n(z)$  of  multigraphs on  $n$ vertices we prove that  

 $$
m_{n}(z)=\frac{1}{n!} \sum_{\alpha  \in S_n^{(2)}} \frac{1}{\det(\text{{\rm \textbf{1}$_m$}}-\alpha  \cdot z)}.
$$


{\bf 2. Algebra of invariants of simple graphs.} 
Let $\mathbb{K}$ be a field of characteristic zero. Denote by  $\mathcal{V}_n$   the set of  undirected graphs on the vertices $\{1,...,n\}$  and
whose edges are weighted in $\mathbb{K}$. A simple graph is a graph
with weights in $\{0,1 \}$ and a multigraph is a graph with
weights in $\mathbb{K}$.
For  any pair   $\{i,j\}$ let    $\textbf{e}_{\{i,j \}}$  be  the simple
graph with one single edge   $\{i,j \} $ and let  $g_{\{i,j \}}\textbf{e}_{\{i,j \}}$  be the graph with one single edge   $\{i,j \} $  and with the weight  $g_{\{i,j \}} \in \mathbb{K}.$ The set  $\mathcal{V}_n$  is the vector space with the basis  $\left\langle \textbf{e}_{\{1,2 \}}, \textbf{e}_{\{1,3 \}}, \ldots \textbf{e}_{\{n-1,n \}} \right\rangle$ of dimension  $m=\displaystyle {n \choose 2}$. Indeed,  any graph can be written    uniquely  as a sum $\sum_{} g_{\{i,j \}} \textbf{e}_{\{i,j \}}.$
Let $\mathcal{V}_n^*$  be the dual space  with dual basis generated by the linear functions  ${x}_{\{i,j \}}$ for which   ${x}_{\{i,j \}}(\textbf{e}_{\{k,l\}})=\delta_{i k} \delta_{j l}.$
The symmetric group  $S_n$ acts on  $\mathcal{V}_n$ and on  $\mathcal{V}_n^*$  by  
$$
\sigma \, \textbf{e}_{\{i,j \}}=\textbf{e}_{\{\sigma(i),\sigma(j) \}}, \sigma^{-1} {x}_{\{i,j \}}={x}_{\{\sigma(i),\sigma(j) \}}.
$$
Let us expand  the action on  the algebra of polynomial functions  $\mathbb{K}[\mathcal{V}_n]=\mathbb{K}[\{{x}_{\{i,j \}}\}].$

We say that a polynomial function   $f \in \mathbb{K}[{x}_{\{i,j \}}]$ of $m$ variables ${x}_{\{i,j \}}$ is a $S_n$-invariant if  $\sigma f=f$ for all  $\sigma \in S_n.$  The   $S_n$-invariants  form a subalgebra   $\mathbb{K}[\mathcal{V}_n]^{S_n}$  which is called the algebra of invariants of the vector space of the weighted graphs in    $n$ vertices. It is clear that there is an isomorphism   $\mathbb{K}[\mathcal{V}_n]^{S_n} \cong \mathbb{K}[{x}_{\{i,j \}}]^{S_n}.$

 For convenience, we introduce a new set of variables:
$$
\{x_1,x_2,\ldots,x_m \}=\{ {x}_{\{1,2 \}}, {x}_{\{1,3 \}}, \ldots, {x}_{\{n-1,n \}}\}.
$$
Then the action of  $S_n$  on the set  $\{ {x}_{\{1,2 \}}, {x}_{\{1,3 \}}, \ldots, {x}_{\{n-1,n \}}\}$ induces its action of the pair group  $S^{(2)}_n$ on the set  $\{x_1,x_2,\ldots,x_m \}.$
We have  
 $$
\mathbb{K}[{x}_{\{i,j \}}]^{S_n}  \cong \mathbb{K}[x_1,x_2,\ldots,x_m]^{S^{(2)}_n}.
$$
In this notation any graph can be written in the way 
$$
g_1 \textbf{e}_{1} +g_2 \textbf{e}_{2}+\cdots+g_m \textbf{e}_{m}, g_i \in \mathbb{K},
$$ 
where  $\textbf{e}_{s}$ is the edge which connect the vertices $\{ i',j' \}$ if the pair  $\{ i',j' \}$ has got   the number  $s.$  Thus,
 in this case the old variable  $x_{\{ i',j' \}}$  corresponds to the new variable  $x_s.$

Since for the simple graphs all its weights are 
 $0,1$ then the reduction of the algebra $\mathbb{K}[\mathcal{V}_n]^{S_n}$  on the set of simple graphs has a simple structure.

Denote by  $\mathcal{V}_n^0$ the set of all simple graphs on $n$ vertices: 
$$
\mathcal{V}_n^0=\left \{ \sum_{i=0}^m {g}_{i}\textbf{e}_{i} \mid {g}_{i} \in \{0,1\} \right \} \subset \mathcal{V}_n,
$$

The corresponding subalgebra of polynomial function $\mathbb{K}[\mathcal{V}_n^0] \subset \mathbb{K}[\mathcal{V}_n]$ is generated by polynomial functions  $x_i$ which on every simple graph only take  values  $1$ or $0.$

Let us consider the ideal $I_m=(x_1^2-x_1,x_2^2-x_2 \ldots,x_m^2-x_m)$ in the algebra  $\mathbb{K}[\mathcal{V}_n]=\mathbb{K}[x_1,x_2,\ldots,x_m]$.  The following statement golds:

\begin{te}\label{1}
$$
\begin{array}{ll}
(i) &\mathbb{K}[\mathcal{V}_n^0]\cong \mathbb{K}[x_1,x_2,\ldots,x_m]/I_m,\\
& \\
(ii) &\mathbb{K}[\mathcal{V}_n^0]^{S^{(2)}_n}\cong \mathbb{K}[x_1,x_2,\ldots,x_m]^{S^{(2)}_n}/I_m.
\end{array}
$$
\end{te}
\begin{proof}  
On the ring of polynomial functions  $\mathbb{K}[x_1,x_2,\ldots,x_m]$ let us introduce a binary  relation  $\sim$: $f \sim g$ if  $f=g$,  considered as functions from $\{ 0,1 \}^m$  to $\mathbb{K}$. 
Obviously that  $x_i^p \sim x_i,$ for all $p \geq 1$. Define an endomorphism   $\gamma: \mathbb{K}[\mathcal{V}_n] \to \mathbb{K}[\mathcal{V}_n^0]$ by the way: $$\gamma(x_i^p)=x_i.$$ It is clear that the kernel of the endomorphism is exactly the ideal $I_m $. Then  
$$ 
\mathbb{K}[\mathcal{V}_n^0]\cong \mathbb{K}[x_1,x_2,\ldots,x_m]/I_m.
$$
Note that the algebra  $\mathbb{K}[\mathcal{V}_n^0]$  is a finite dimensional vector space of the dimension  $2^m$  with the basis 
$$
1, x_1,x_2,\ldots,x_n, x_1 x_2, x_1 x_2, \ldots, x_{m-1}x_m, \ldots, x_1 x_2\cdots x_m.
$$

$(ii)$ It is enough to prove that  $\gamma$ commutes with the action of the group $S_n^{(2)}$. Without  lost of generality  it  is sufficient to check on the monomials. For arbitrary element $\sigma \in S_n^{(2)}$ and for arbitrary  monomial   $x_1^{k_1}x_2^{k_2}\cdots x_s^{k_s}$, $s \leq m$ we  have
\begin{gather*}
\gamma(\sigma^{-1}(x_1^{k_1}x_2^{k_2}\cdots x_s^{k_s}))=\gamma(x_{\sigma(1)}^{k_1}x_{\sigma(2)}^{k_2}\cdots x_{\sigma(s)}^{k_s})=x_{\sigma(1)}x_{\sigma(2)}\cdots x_{\sigma(s)}=\\
=\sigma^{-1}(x_1 x_2  \cdots x_s )=\sigma^{-1}\left( \gamma(x_1^{k_1}x_2^{k_2}\cdots x_s^{k_s}) \right).
\end{gather*}
\end{proof}

If we know the algebra of invariants  $\mathbb{K}[\mathcal{V}_n]^{S^{(2)}_n}$  then we are able to find the algebra of invariants  $\mathbb{K}[\mathcal{V}_n^0]^{S^{(2)}_n}$ of simple graphs. Indeed, if the invariants  $f_1, f_2, \ldots, f_s$  generate the algebra $\mathbb{K}[\mathcal{V}_n]^{S^{(2)}_n}$ then the surjectivity  of $\gamma$ implies  that  the invariants $\gamma(f_1), \gamma(f_2), \ldots, \gamma(f_s)$ generate the algebra  $\mathbb{K}[\mathcal{V}_n^0]^{S^{(2)}_n}.$

\textbf{Example.} Let us consider the case  $n=4$.  The algebra of invariants  $\mathbb{K}[\mathcal{V}_4]^{S_4^{(2)}}$ is well known, see  \cite{HST},
and its minimal generating system consists of the following 9  invariants: 
\begin{gather*}
R(x_1)=\frac{1}{6}(x_1+x_2+x_3+x_4+x_5+x_6),
R(x_1^2)=\frac{1}{6}(x_1^2+x_2^2+x_3^2+x_4^2+x_5^2+x_6^2),\\
R(x_1 x_6)=\frac{1}{3}(x_{{1}}x_{{6}}+x_{{2}}x_{{5}}+\,x_{{3}}x_{{4}}),
R(x_1^3)=\frac{1}{6}(x_1^3+x_2^3+x_3^3+x_4^3+x_5^3+x_6^3),\\
24 R(x_1^2 x_2)={x_{{1}}}^{2}x_{{2}}+{x_{{1}}}^{2}x_{{3}}+{x_{{2}}}^{2}x_{{1}}+{x_{{2}
}}^{2}x_{{3}}+{x_{{3}}}^{2}x_{{1}}+{x_{{3}}}^{2}x_{{2}}+{x_{{1}}}^{2}x
_{{4}}+{x_{{1}}}^{2}x_{{5}}+{x_{{4}}}^{2}x_{{1}}+\\+{x_{{4}}}^{2}x_{{5}}+
{x_{{5}}}^{2}x_{{1}}+{x_{{5}}}^{2}x_{{4}}+{x_{{2}}}^{2}x_{{4}}+{x_{{2}
}}^{2}x_{{6}}+{x_{{4}}}^{2}x_{{2}}+{x_{{4}}}^{2}x_{{6}}+{x_{{6}}}^{2}x
_{{2}}+{x_{{6}}}^{2}x_{{4}}+{x_{{3}}}^{2}x_{{5}}+\\+{x_{{3}}}^{2}x_{{6}}+
{x_{{5}}}^{2}x_{{3}}+{x_{{5}}}^{2}x_{{6}}+{x_{{6}}}^{2}x_{{3}}+{x_{{6}
}}^{2}x_{{5}}\\
R(x_1 x_2 x_3)=\frac{1}{4}(x_{{1}}x_{{3}}x_{{2}}+x_{{1}}x_{{5}}x_{{4}}+x_{{2}}x_{{
6}}x_{{4}}+x_{{3}}x_{{6}}x_{{5}}
),\\
R(x_1^4)=\frac{1}{6}(x_1^4+x_2^4+x_3^4+x_4^4+x_5^4+x_6^4),R(x_1^5)=\frac{1}{6}(x_1^5+x_2^5+x_3^5+x_4^5+x_5^4+x_6^5),\\
24 R(x_1^3 x_2)={x_{{1}}}^{3}x_{{2}}+{x_{{1}}}^{3}x_{{3}}+{x_{{2}}}^{3}x_{{1}}+{x_{{2}
}}^{3}x_{{3}}+{x_{{3}}}^{3}x_{{1}}+{x_{{3}}}^{3}x_{{2}}+{x_{{1}}}^{3}x
_{{4}}+{x_{{1}}}^{3}x_{{5}}+{x_{{4}}}^{3}x_{{1}}+\\+{x_{{4}}}^{3}x_{{5}}+
{x_{{5}}}^{3}x_{{1}}+{x_{{5}}}^{3}x_{{4}}+{x_{{2}}}^{3}x_{{4}}+{x_{{2}
}}^{3}x_{{6}}+{x_{{4}}}^{3}x_{{2}}+{x_{{4}}}^{3}x_{{6}}+{x_{{6}}}^{3}x
_{{2}}+{x_{{6}}}^{3}x_{{4}}+{x_{{3}}}^{3}x_{{5}}+\\+{x_{{3}}}^{3}x_{{6}}+
{x_{{5}}}^{3}x_{{3}}+{x_{{5}}}^{3}x_{{6}}+{x_{{6}}}^{3}x_{{3}}+{x_{{6}
}}^{2}x_{{5}}.
\end{gather*}
Here  $$R=\frac{1}{n!} \sum_{g \in S_n^{(2)}} g$$  is the Reinolds group action averaging operator which  is a projector from $\mathbb{K}[\mathcal{V}_n]$ into $\mathbb{K}[\mathcal{V}_n]^{S_n^{(2)}}.$

We  have  that  $$\gamma(R(x_1^5))=\gamma(R(x_1^4))=\gamma(R(x_1^3))=\gamma(R(x_1^2))=R(x_1),\gamma(R(x_1^2 x_2))=R(x_1 x_2).$$ 
Therefore, the algebra of invariants  $\mathbb{K}[\mathcal{V}_4^0]^{S_4^{(2)}}$ of simple graphs on $n$ vertices is generated by the $4$ invariants:
$$
R(x_1), R(x_1 x_6), R(x_1 x_2), R(x_1,x_2,x_3).
$$

So far, the algebra of invariants $\mathbb{K}[\mathcal{V}_n]^{S^{(2)}_n}$ is calculated only for  $n\leq 5$,  see \cite{NT}. 


{\bf 3.  The Poincar\'e series of the algebra  $\mathbb{K}[\mathcal{V}_n^0]^{S^{(2)}_n}.$}
Let us consider  the algebra    $\mathbb{K}[\mathcal{V}_n^0]$  as a vector space. Then    the following  decomposition into  the direct  sum of its subspaces holds:
$$
\mathbb{K}[\mathcal{V}_n^0]=(\mathbb{K}[\mathcal{V}_n^0])_0+(\mathbb{K}[\mathcal{V}_n^0])_1+\cdots+(\mathbb{K}[\mathcal{V}_n^0])_m,
$$
where  $(\mathbb{K}[\mathcal{V}_n^0])_i$ is the vector space generated by the elements 
$$ x_1^{\varepsilon_1}x_2^{\varepsilon_2} \cdots x_m^{\varepsilon_m}, \varepsilon_1+\varepsilon_2+\cdots+\varepsilon_m=i, \text{  where } \varepsilon_k=0 \text{  or   }  \varepsilon_k=1.
$$

Also, for the  algebra 
   $\mathbb{K}[\mathcal{V}_n^0]^{S^{(2)}_n}$  the decomposition holds: 
$$
\mathbb{K}[\mathcal{V}_n^0]^{S^{(2)}_n}=(\mathbb{K}[\mathcal{V}_n^0]^{S^{(2)}_n})_0+(\mathbb{K}[\mathcal{V}_n^0]^{S^{(2)}_n})_1+\cdots+(\mathbb{K}[\mathcal{V}_n^0]^{S^{(2)}_n})_m.
$$
Since, the Reynolds operator is a projector which save the degree of a polynomial  then  the component   $(\mathbb{K}[\mathcal{V}_n^0]^{S^{(2)}_n})_i$  is generated by the following elements 
$$ R(x_1^{\varepsilon_1}x_2^{\varepsilon_2} \cdots x_m^{\varepsilon_m}).
$$

Particularly, we have  that $(\mathbb{K}[\mathcal{V}_n^0]^{S^{(2)}_n})_0=\mathbb{K}$. Also,  the component  $(\mathbb{K}[\mathcal{V}_n^0]^{S^{(2)}_n})_m$ has the dimension  $1$  and it is generated by the polynomial  $R(x_1 x_2 \cdots x_m) =  x_1 x_2 \cdots x_m.$ 
  
Let us now give an interpretation of  $\dim(\mathbb{K}[\mathcal{V}_n^0]^{S^{(2)}_n})_i$  in terms of the graph theory. The following important theorem holds.
\begin{te}\label{t2}
The dimension   $\dim(\mathbb{K}[\mathcal{V}_n^0]^{S^{(2)}_n})_i$ equal to  the number of non-isomorphic simple graphs with $n$ vertices and $i$ edges.
\end{te}
\begin{proof} 
The $S^{(2)}_n$-module  $(\mathbb{K}[\mathcal{V}_n^0]/I_m)_i$ is generated by the monomials  $$x_1^{\varepsilon_1}x_2^{\varepsilon_2} \cdots x_m^{\varepsilon_m}, \varepsilon_1+\varepsilon_2+\cdots+\varepsilon_m=i, \varepsilon_k=0 \text{ or }  \varepsilon_k=1.$$ 
Since the group  $S^{(2)}_n$ is finite  then the module  $(\mathbb{K}[\mathcal{V}_n^0]/I_m)_i$ is decomposed into the direct sum of its  irreducible $S^{(2)}_n$-submodules:
$$
(\mathbb{K}[\mathcal{V}_n^0])_i=M_1 \oplus M_2 \oplus \cdots \oplus M_p.
$$
Each of these submodules  has a basis generated by  the monomials of the form  $x_1^{\varepsilon_1}x_2^{\varepsilon_2} \cdots x_m^{\varepsilon_m}$. Let us choose for each  $M_i$ the corresponding basis monomials  $m_1, m_2,\ldots,m_p$ and consider the  invariants  $R(m_1), R(m_2),\ldots,R(m_p).$ By the construction they are different and linearly independent. Therefore the component  $(\mathbb{K}[\mathcal{V}_n^0]^{S^{(2)}_n})_i$ is the sum of one-dimensional  $S^{(2)}_n$-submodules 
$$
(\mathbb{K}[\mathcal{V}_n^0]^{S^{(2)}_n})_i=\left\langle R(m_1)\right\rangle +\left\langle R(m_2)\right\rangle + \cdots + \left\langle R(m_p)\right\rangle,
$$
and  $\dim(\mathbb{K}[\mathcal{V}_n^0]^{S^{(2)}_n})_i=p$ for some $p.$ To each of monomial  $m_1, m_2,\ldots,m_p$ assign a simple graph in the following way: if  $m_i=x_1^{\varepsilon_1}x_2^{\varepsilon_2} \cdots x_m^{\varepsilon_m}$ then the corresponding  simple graph has the form  
$$
G_{m_i}={\varepsilon_1}\textbf{e}_{1}+{\varepsilon_2}\textbf{e}_{2}+ \cdots +{\varepsilon_m}\textbf{e}_{m}.
$$
Since the monomials  $m_1, m_2,\ldots,m_p$ belong to the different irreducible  $S^{(2)}_n$-modules  then  $G_{m_i}$ are non-isomorphic and they exhausted all the possible classes of isomorphic classes of isomorphic graphs with $n$ vertices and $i$ edges.
\end{proof}

Let us recall that the ordinary  generating function for the sequence   $\dim(\mathbb{K}[\mathcal{V}_n^0]^{S^{(2)}_n})_i$ 
$$
\mathcal{P}(\mathbb{K}[\mathcal{V}_n^0]^{S^{(2)}_n},z)=\sum_{i=0}^m \dim(\mathbb{K}[\mathcal{V}_n^0]^{S^{(2)}_n})_i \cdot z^i.
$$
is called the Poincar\'e series of the algebra $\mathbb{K}[\mathcal{V}_n^0]^{S^{(2)}_n}.$

The  Theorem \ref{t2} implies that 

$$
g_{n}(z)=\mathcal{P}(\mathbb{K}[\mathcal{V}_n^0]^{S^{(2)}_n},z).
$$

In the following theorem we derived   explicit formulas for the series 
 $\displaystyle  \mathcal{P}(\mathbb{K}[\mathcal{V}_n^0]^{S^{(2)}_n},z)$ and $\displaystyle \mathcal{P}(\mathbb{K}[\mathcal{V}_n]^{S^{(2)}_n},z).$

\begin{te}
Let the group  $S_n^{(2)}$ be  realized as   $m \times m $ matrices. Then  
$$
\begin{array}{ll}
(i) & \displaystyle  \mathcal{P}(\mathbb{K}[\mathcal{V}_n^0]^{S^{(2)}_n},z)=\frac{1}{n!} \sum_{\alpha \in S^{(2)}_n} \frac{\det(\text{{\rm \textbf{1}$_m$}}-\alpha \cdot z^2)}{\det(\text{{\rm \textbf{1}$_m$}}-\alpha \cdot z)},\\
(ii) & \displaystyle \mathcal{P}(\mathbb{K}[\mathcal{V}_n]^{S^{(2)}_n},z)=\frac{1}{n!} \sum_{\alpha \in S^{(2)}_n} \frac{1}{\det(\text{{\rm \textbf{1}$_m$}}-\alpha \cdot z)}.
\end{array}
$$
here   $\text{{\rm \textbf{1}$_m$}}$ is the unit  $m \times m $ matrix. 
\end{te}
\begin{proof}

$(i)$  The vector space $(\mathbb{K}[\mathcal{V}_n^0])_1$  has  the basis  $\left\langle x_1,x_2, \ldots,x_m \right\rangle$ and a permutation    $\alpha \in S^{(2)}_n$  acts  on the  $(\mathbb{K}[\mathcal{V}_n^0])_1$ by permutation of the basis vectors. Denote this  linear operator by $A_{\alpha}$.     Let us expand this operator on the component   $(\mathbb{K}[\mathcal{V}_n^0]^{S^{(2)}_n})_k$ as endomorphism    and denote it  by  $A_{\alpha}^{(k)}$. Since $A_{\alpha}^{(k)}$ is endomorphism and acts as a permutation of the basis vectors of  $(\mathbb{K}[\mathcal{V}_n^0]^{S^{(2)}_n})_k$  then the action of the operator  $A_{\alpha}^{(k)}$  is defined correctly.

Let a permutation $\alpha$   be written uniquely as a product 
of disjoint cycles  and let $j_i(\alpha)$ be the number 
of cycles of length $i$ in the disjoint cycle decomposition of $\alpha$.

Now find the track of the operator  $A_{\alpha}^{(k)}.$
\begin{lm}
$$
{\rm Tr}(A_{\alpha}^{(i)})=\sum_{\beta_1+2 \beta_2+\cdots+m \beta_m=i} \binom{j_1(\alpha)}{\beta_1}\binom{j_2(\alpha)}{\beta_2} \cdots \binom{j_i(\alpha)}{\beta_i}.
$$
\end{lm}
\begin{proof}
Since the operator  $A_{\alpha}^{(i)}$ acts by permutations of the basis vectors of the vector space  $(\mathbb{K}[\mathcal{V}_n^0])_i$  that its track equal to the numbers of its fixed point.

For  $i=1$ we have  $(\mathbb{K}[\mathcal{V}_n^0])_1=\left\langle x_1,x_2, \ldots,x_m\right\rangle$  and $A_{\alpha}^{(1)}(x_s)=x_{\alpha^{-1}(s)}$. Thus   ${\rm Tr}(A_{\alpha}^{(1)})=j_1(\alpha).$

For  $i=2$ let us find out the number of fixed points of the operator   $A_{\alpha}^{(2)}$ which acts on the vector space  $(\mathbb{K}[\mathcal{V}_n^0])_2$  with the basis vectors  $x_i x_j, i<j.$ 
 An arbitrary pair of fixed points of the operator  $A_{\alpha}$ form one fixed point for the operator  $A_{\alpha}^{(2)}$. Thus we get  $\binom{j_1(\alpha)}{2}$ such points. Also, every transposition define one fixed point. Therefore 
$$
{\rm Tr}(A_{\alpha}^{(2)})=\displaystyle \binom{j_1(\alpha)}{2}+j_2(\alpha).
$$

All $j_1(\alpha)$ fixed points of the permutation  $\alpha$  generates $\binom{j_1(\alpha)}{3}$ fixed points of the operator  $A_{\alpha}^{(3)}$. Every fixed point of  $A_{\alpha}$ together with  $j_2(\alpha)$ transposition generate one fixed point of  $A_{\alpha}^{(3)}$. At  last, each any of  $3$-cycle  of $\alpha$  generates one fixed point for  $A_{\alpha}^{(3)}$. Then
$$
{\rm Tr}(A_{\alpha}^{(3)})=\displaystyle \binom{j_1(\alpha)}{3}+j_1(\alpha) j_2(\alpha)+j_3(\alpha).
$$
Analogously
$$
{\rm Tr}(A_{\alpha}^{(4)})=\displaystyle \binom{\alpha_1}{4}+\binom{\alpha_1}{2} \alpha_2+\binom{\alpha_2}{2}+\alpha_1 \alpha_3+\alpha_4.
$$
In the general case  any partition of  $i$
$$
\beta_1+2 \beta_2+\cdots+m \beta_m=i.
$$
generates 
 $$
 \binom{j_1(\alpha)}{\beta_1}\binom{j_2(\alpha)}{\beta_2} \cdots \binom{j_m(\alpha)}{\beta_m}
 $$
fixed points of the operator $A_{\alpha}^{(i)}.$ Therefore
$$
{\rm Tr}(A_{\alpha}^{(i)})=\sum_{\beta_1+2 \beta_2+\cdots+m \beta_m=i} \binom{j_1(\alpha)}{\beta_1}\binom{j_2(\alpha)}{\beta_2} \cdots \binom{j_m(\alpha)}{\beta_m}.
$$
\end{proof}

\begin{lm}
$$
\sum_{i=0}^{m} {\rm Tr}(A_{\alpha}^{(i)})z^i=(1+z)^{j_1(\alpha)} (1+z^2)^{j_2(\alpha)} \cdots (1+z^m)^{j_m(\alpha)}.
$$
\end{lm}
\begin{proof} We  have 

\begin{gather*}
\sum_{i=0}^{m}{\rm Tr}(A_{\alpha}^{(i)}) z^i=\sum_{\beta_1+2 \beta_2+\cdots+m \beta_m=i} \binom{j_1(\alpha)}{\beta_1}\binom{j_2(\alpha)}{\beta_2} \cdots \binom{j_m(\alpha)}{\beta_i}z^i=\\=
\sum_{\beta_1+2 \beta_2+\cdots+m \beta_m=i} \binom{j_1(\alpha)}{\beta_1}\binom{j_2(\alpha)}{\beta_2} \cdots \binom{j_i(\alpha)}{\beta_i}z^{\beta_1+2 \beta_2+\cdots+m \beta_m}=\\=\sum_{\beta_1+2 \beta_2+\cdots+m \beta_m=i} \binom{j_1(\alpha)}{\beta_1}z^{\beta_1}\binom{j_2(\alpha)}{\beta_2} (z^2)^{\beta_2}\cdots \binom{j_i(\alpha)}{\beta_i}(z^m)^{\beta_m}=\\=
\left(\sum_{\beta_1=0}^m \binom{j_1(\alpha)}{\beta_1}z^{\beta_1} \right) \left(\sum_{\beta_2=0}^m \binom{j_2(\alpha)}{\beta_2} (z^2)^{\beta_2} \right) \cdots \left(\sum_{\beta_m=0}^m  \binom{j_m(\alpha)}{\beta_m}(z^m)^{\beta_m} \right)=\\
=(1+z)^{j_1(\alpha)} (1+z^2)^{j_2(\alpha)} \cdots (1+z^m)^{j_m(\alpha)}.
\end{gather*}

\end{proof}

\begin{lm}
$$
\sum_{i=0}^{m} {\rm Tr}(A_{\alpha}^{(i)})z^i=\frac{\det(\text{{\rm \textbf{1}$_m$}}-A_{\alpha} \cdot z^2)}{\det(\text{{\rm \textbf{1}$_m$}}-A_{\alpha} \cdot z)}.
$$
\end{lm}
\begin{proof}
Let  $\lambda_1,\lambda_2, \ldots, \lambda_m$ be the eigenvalues of the operator  $A_{\alpha}.$  Since   $A_{\alpha}^{m!}$  is the identity matrix  then all  eigenvalues  $\lambda_i$  are roots of  unity   of orders   $j_1(\alpha),j_2(\alpha) \ldots, j_m(\alpha)$. Therefore the characteristic polynomial of the operator   $A_{\alpha}$ has the form 
$$
\det(\text{{\rm \textbf{1}$_m$}}-A_{\alpha} z)=(1-\lambda_1 z) (1-\lambda_2 z)\cdots (1-\lambda_n z)=(1-z)^{j_1(\alpha)}(1-z^2)^{j_2(\alpha)} \cdots (1-z^m)^{j_m(\alpha)}.
$$
Now 
\begin{gather*}
\sum_{i=0}^{m} {\rm Tr}(A_{\alpha}^{(i)})z^i=(1+z)^{j_1(\alpha)} (1+z^2)^{j_2(\alpha)} \cdots (1+z^m)^{j_m(\alpha)}=\\=
\frac{(1-z^2)^{j_1(\alpha)}}{(1-z)^{j_1(\alpha)}}\, \frac{(1-z^4)^{j_2(\alpha)}}{(1-z^2)^{j_2(\alpha)}} \cdots \frac{(1-z^{2m})^{j_m(\alpha)}}{(1-z^m)^{j_m(\alpha)}}=\frac{\det(\text{{\rm \textbf{1}$_m$}}-A_{\alpha} \cdot z^2)}{\det(\text{{\rm \textbf{1}$_m$}}-A_{\alpha} \cdot z)}.
\end{gather*}
\end{proof}

Let us find the dimension of  $S^{(2)}_n$-invariant subspace  $(\mathbb{K}[\mathcal{V}_n^0])_i$. 

\begin{lm}
$$
\dim (\mathbb{K}[\mathcal{V}_n^0]^{S^{(2)}_n})_i=\frac{1}{n!}\sum_{\alpha \in G} {\rm Tr}(A_{\alpha}^{(i)}). 
$$
\end{lm}

\begin{proof} The dimension of the subspace  $(\mathbb{K}[\mathcal{V}_n^0]^{S^{(2)}_n})_i$ is equal to the number of eigenvectors that correspond to the eigenvalue $1$ and which are common eigenvectors for all operators   
 $A_{\alpha}^{(i)}$.
  Consider the average matrix 
$$
P^{(i)}=\frac{1}{|G|}\sum_{g \in G}A_{\alpha}^{(i)}.
$$

Since  the Reynolds operator is a projector from  $(\mathbb{K}[\mathcal{V}_n^0])_i$ into  $(\mathbb{K}[\mathcal{V}_n^0]^{S^{(2)}_n})_i$ then it has the  only eigenvalues  $1$  and  $0.$
Therefore the dimension of the space   $(\mathbb{K}[\mathcal{V}_n^0]^{S^{(2)}_n})_i$ is equal to the track of the matrix $P^{(i)}$.

\end{proof} 
Taking into account the  lemmas stated above we have 

\begin{gather*}
\mathcal{P}(\mathbb{K}[\mathcal{V}_n^0]^{S^{(2)}_n},z)=\sum_{i=0}^{m}\dim (\mathbb{K}[\mathcal{V}_n^0]^{S^{(2)}_n})_i z^i=\frac{1}{n!}\sum_{i=0}^{m}\left( \sum_{\alpha \in S^{(2)}_n} {\rm Tr}(A_{\alpha}^{(i)})  \right) z^i=\\=\frac{1}{n!}\sum_{g \in S^{(2)}_n}\left( \sum_{i=0}^{m}{\rm Tr}(A_{\alpha}^{(i)}) z^i  \right)
=\frac{1}{n!}\sum_{\alpha \in S^{(2)}_n}\frac{\det(\text{{\rm \textbf{1}$_m$}}-\alpha \cdot z^2)}{\det(\text{{\rm \textbf{1}$_m$}}-\alpha \cdot z)}.
\end{gather*}

$(ii)$  It is the Molien formula for the Poincar\'e series of the algebra invariants of the group  $S^{(2)}_n.$
\end{proof}

\end{document}